\documentclass[12pt]{amsart}
\usepackage{amsmath}
\usepackage{amssymb,amscd}
\usepackage[T1]{fontenc}
\usepackage[latin1]{inputenc}
\usepackage{tikz}
\usepackage{tikz-cd}

\usepackage{color}

\usepackage{graphicx}

\usepackage{psfrag}

\numberwithin{equation}{section}

\newtheorem{defn}{Definition}[section]
\newtheorem{theorem}{Theorem}[section]

\newtheorem{example}[theorem]{Example}

\newtheorem{lemma}[theorem]{Lemma}

\newtheorem{remark}[theorem]{Remark}

\def \begineq{\begin{equation}}
\def \endeq{\end{equation}}

\def \bb{\mathbb}

\def \bs{\boldsymbol}

\def \CC{{\bb{C}}}

\def \ZZ{{\bb{Z}}}

\def \({\left(}
\def \){\right)}
\def \<{\langle}
\def \>{\rangle}

\begin{document}
\title[Heegaard Floer invariants for cyclic 3-orbifolds]{Heegaard Floer invariants for cyclic 3-orbifolds}

\author[S. Ganguli]{Saibal Ganguli}
\address{Swamina International Pvt. Ltd., Kripa Bhavan, EP Block, Sector V, Salt Lake Electronic Complex, Kolkata 700091, India}
\email{saibalgan@gmail.com}

\author[M. Poddar]{Mainak Poddar}
\address{Department of Mathematics, 
IISER Pune, 
Dr. Homi Bhabha Road, Pashan, Pune 411008, India}
\email{mainak@iiserpune.ac.in}



 \begin{abstract}
 We define a notion of Heegaard Floer homology for three dimensional orbifolds with arbitrary cyclic singularities, generalizing the recent work of Biji Wong where the singular locus is assumed to be connected.
 \end{abstract}
\maketitle

 \section{Introduction}
 Heegaard Floer homology, introduced  by Ozsvath and Szabo in \cite{{OZ}}, provides a package of powerful invariants for any $3$-manifold by cleverly associating a certain symplectic manifold (not unique) to it based on its Heegaard decomposition. The Lagrangian Floer homology of the associated sympletic manifold turns out to be an invariant of the $3$-manifold. Subsequently, a {\it cylindrical} reformulation of the theory was obtained by Lipshitz \cite{Lip} where the dimension of the associated symplectic manifold is four, making the theory more tractable in terms of visualization and computation. A generalization  of the theory for 3-manifolds with 
 connected boundary called  bordered Heegaard Floer homology was developed by Ozsvath,
 Lipschitz and Thurston \cite{LOT1, LOT2}. Their approach heavily uses the algebraic machinery of $\mathcal{A}_{\infty}$ modules and bimodules. 
 To each parametrized closed surface $F$ is associated an algebra, and to each $3$-manifold with boundary $F$ are associated two different kinds of module over the algebra, called the type-$A$ and type-$D$ modules.  If a closed $3$-manifold $Y =Y_1\cup_F Y_2 $ is decomposed along a separating surface F, then the
 Heegaard Floer homology $\widehat{HF}(Y)$ is isomorphic to the homology of a {\it box} or {\it derived} tensor product of the type-$A$ module of $Y_1$ and the type-$D$ module of $Y_2$. 
 
 Bordered Heegaard Floer theory was recently utilized by Wong \cite{BW} to introduce an invariant for $3$-orbifolds with cyclic singularity along a knot. The idea is to decompose the $3$-orbifold into a $3$-manifold with a torus boundary and an orbifold solid torus. Then the type-$D$ module of the orbifold part is taken to be essentially the type-$D$ module of its natural smooth branched cover. This is then joined with the type-$A$ module of the bordered manifold part by a box tensor product. This yields the required invariant. Note that extensive work has been done previously in providing homological invariants to $3$-dimensional orbifolds using gauge theoretic invariants such as Floer's instanton homology, see \cite{CS, KM}. 
 
   Three dimensional orbifolds are spaces which are locally quotients of $\mathbb{R}^{3}$ by  finite subgroups of $SO(3)$. Such an  orbifold is called cyclic if the local groups are cyclic. The orbifold singular locus of a cyclic $3$-orbifold is a disjoint union of knots. However, the underlying topological space of such an orbifold is homeomorphic to a manifold. The purpose of this article is to define the hat version of Heegaard Floer homology, $\widehat{HF}$, for arbitrary cyclic $3$-orbifolds, generalizing  Wong's work \cite{BW} where the orbifold singular locus is assumed to be connected. There is a challenge here as bordered Heegaard Floer homology is not well-developed  for manifolds with multiple boundary components. We strive to overcome this by initially smoothening out the orbifold singularities
   and reintroducing their algebraic contributions at a later stage in a sequential manner. A critical idea used by us is to sequentially modify the $A$-module structure of the bordered manifold part(s), so that it reflects the singularity of the component of the orbifold locus we deal with at each intermediate stage. The final stage is similar to Wong's construction. We obtain a smooth invariant of the cyclic $3$-orbifold which, in general, depends on the ordering of the components of its singular locus. This dependence can be relaxed in some very special cases, see Lemma \ref{equiv}, Example 
   \ref{example: lens space} and Example \ref{sub:r-surg}.
   
   The paper is organized as follows. We review the basics of bordered Heegaard Floer theory in Section 2. In the next section, we briefly review Wong's construction of Heegaard Floer homology for $3$-orbifolds with cyclic singularity along a single knot. We extend this theory to singularity along two knot components in Section 4. This section describes the main ideas and results of this work. Finally, we outline the construction of the invariant for arbitrary singular locus in Section 5.   
  
  \subsection*{Acknowledgements}
The authors would like to thank Robert Lipshitz, Dylan Thurston, Peter Ozsvath, and Biji Wong for helpful email correspondence. They also thank the referee for helpful comments and suggestions. The first author would like to thank IISER Pune for its support and hospitality during a three month visit in 2021  when the initial part of work was done, and Bhaskaracharya Pratisthana Pune for a postdoctoral fellowship where significant progress was made. The research of the second author was supported in part by a SERB MATRICS Grant: MTR/2019/001613.
 
 \section{Bordered Floer Homology}
 
 In this section we give a brief overview of bordered Floer homology for three manifolds with boundary based on the treatment in \cite{HRW, LOT1, LOT2, BW}. Our main goal is to set up the notation. We concentrate on the torus boundary case as it is sufficient for our purposes.
 
 \subsection{$A$ and $D$ structures}\label{A and D}
 
 Let $\mathcal{A}$ be the unital path algebra over $\mathbb{Z}_2$ associated to the quiver in  Figure \ref{figure1} (cf. Figure 1 in \cite{BW})  modulo the relations $\rho_2 \rho_1$, $\rho_3 \rho_2$. In other words, the
  composition of paths (between distinct vertices) is nontrivial only when their indices increase. As a $\mathbb{Z}_2$ vector space $\mathcal{A}$ is generated by eight elements:  two idempotents $i_1$
 and $i_2$, and six {\it Reeb} elements $\rho_1$, $\rho_2$, $\rho_3$, $\rho_{12} := \rho_1 \rho_2$, $\rho_{23} :=\rho_2 \rho_3$ and $\rho_{123} := \rho_1 \rho_2 \rho_3$. The multiplicative identity  of $\mathcal{A}$ is given by $\mathbf{1}:= i_1 + i_2$. We denote the multiplication operation of $\mathcal{A}$ by $\mu: \mathcal{A} \otimes \mathcal{A} \to \mathcal{A} $. Also, denote the subalgebra  of $\mathcal{A}$ generated by $i_1$ and $i_2$ by $\mathcal{I}$.

     


\begin{figure}
\begin{tikzpicture}[scale=.5]

\draw [->] (5,-2)--(0,-2);
\draw[thick, ->] (0,-1.8) arc (180:0:2.5cm);
\draw[thick, ->] (0, -2.2) arc (180:360:2.5cm);

\node at (2.5, 0.2) {$\rho_1$};
\node at (2.5,-1.5) {$\rho_2$};
\node at (2.5, -3.8) {$\rho_3$};
\node at (-0.5, -2) {$i_1$};
\node at (5.5, -2) {$i_2$};

\end{tikzpicture}
\caption{ Quiver related to the path algebra $\mathcal{A}$ }
\label{figure1}
\end{figure}
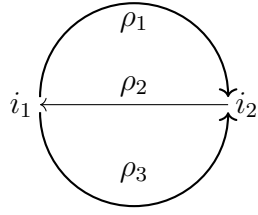

 A (left) type $D$ structure over $\mathcal{A}$ is a pair $(N,\delta_1)$ where $N$ is a finite dimensional $\mathbb{Z}_2$-vector space   equipped with a (left) action by $\mathcal{I}$ so that
 \begin{equation}
  N=i_1{N} \oplus i_2{N} \,
 \end{equation}
 and  $\delta_1: N \rightarrow  \mathcal{A} \otimes_{\mathcal{I}} N$ is a map that satisfies the following relation
 \begin{equation}\label{eq:delta1}
  (\mu \otimes id_N) \circ (id_\mathcal{A} \otimes \delta_1) \circ \delta_1=0 \, .
 \end{equation}
 We can inductively define maps
$\delta_k$ such that $\delta_k = (id_{{\mathcal{A}}^{{\otimes}^{k-1}}} \otimes \delta_1 ) \circ \delta_{k-1}$. The type $D$ structure is said to be {\it bounded} if $\delta_k =0$ for  all sufficiently large k.

It is convenient to represent type $D$ structures $(N,\delta_1)$  by decorated directed graphs. Construct a basis for $N$ by fixing bases for each subspace $i_{*}N$. Assign a vertex to each basis element. Decorate a vertex  with a $\bullet$ sign if the basis element belongs to $i_{1}N$, otherwise   decorate it with a $\circ$ sign.  When basis elements $x$ and $y$ satisfy the condition: $\rho_{I} \otimes y$ is a summand  of $\delta_1(x)$ with $\rho_{I} \in\{\rho_{\phi}=\mathbf{1},\rho_1,\rho_2,\rho_3,\rho_{12},\rho_{23},\rho_{123}\}$, introduce a directed edge from
vertex $x$ to vertex $y$, and label the edge with $\rho_{I}$ . The relation \eqref{eq:delta1} on $\delta_{1}$ translates to a condition on the graph: For any directed path of length $2$, the labels have product $0$ in $\mathcal{A}$. The higher maps $\delta_k$ can be encoded using directed paths of length $k$.  Conversely, a type $D$ structure can also be recovered from a graph.

A {\em (right) type $A$ structure over} $\mathcal{A}$ is a pair
$(M, {\{m_k\}}^{\infty}_{k=1})$
consisting of a finite-dimensional $\mathbb{Z}_2$-vector space $M$ with a right action of $\mathcal{I}$ so that
\begin{equation}
 M= M i_1 \oplus M i_2
\end{equation}
 as a vector space,
 and maps  
 \begin{equation}
  m_k: M \otimes_{\mathcal I} \mathcal{A}^{\otimes k-1}  \rightarrow M 
 \end{equation}
that satisfy the relation, 
\begin{equation}\label{compat}
	\begin{array}{l}
  {\sum}^{k}_{j=1} m_{k-j+1}(m_{j}(x,a_1\otimes a_2 \ldots \otimes a_{j-1}),a_j,\ldots, a_{k-1}) \\
   \vspace{0.1cm}\\
   +{\sum}^{k-2}_{j=1}m_{k-1}(x,a_1 \otimes\ldots a_j\otimes a_j a_{j+1}\otimes a_{j+2}\ldots  \otimes a_{k-1}) = 0 \,,
  \end{array}
  \end{equation}
for all $x \in M$, $k \in \mathbb{N}$ and $a_1, \ldots, a_{k-1} \in \mathcal{A}$

 A type $A$ structure  is {\em bounded} if $m_k = 0$ for all sufficiently large values of $k$. 
 If $(M, {\{m_k\}}^{\infty}_{k=1}$ is a type $A$ structure over 
 $\mathcal{A}$ and  $(N,\delta_1)$ is a type $D$ structure over $\mathcal{A}$ such that at least one of them is bounded, then we can define a box tensor product
  $M \boxtimes N$, a $\mathbb{Z}_2$-chain complex $(M \otimes_{\mathcal{I}} N, {\delta}^{\boxtimes})$ where $\delta^{\boxtimes}: M\otimes_{\mathcal{I}} N\rightarrow M \otimes_{\mathcal{I}} N$ given by
 \begin{equation}  
 \delta^{\boxtimes}(x \otimes y)={\sum}^{\infty}_{k=0}(m_{k+1} \otimes id_N )(x \otimes {\delta}_{k}(y)) \,. 
 \end{equation}


\subsection{Bordered three manifolds}

We closely follow the description in \cite[Section 2.2.2]{BW}, but include the definitions and notations for the convenience of the reader. 

 A bordered 3-manifold is defined to be a pair $(Y, \phi)$
where $Y$ is a compact, connected, oriented 3-manifold with connected boundary, and $\phi$ is a homeomorphism from a fixed model surface $F$ to the boundary of $Y$. Two bordered
3-manifolds $(Y_1, \phi_1)$ and $(Y_2, \phi_2)$ are said to be equivalent if there exists an orientation-preserving
homeomorphism $\psi \colon  Y_1 \rightarrow Y_2$ satisfying $\phi_2 = \psi|_{\partial} \circ \phi_1$.

We focus our attention on the case when $\partial Y$ is a torus. In this case, the surface $F$ is an oriented
torus associated to the pointed matched circle $\mathcal{Z}$ in Figure \ref{figure2}  (cf. Figure 2 \cite{BW}). This association is  explained in \cite[Definition 3.17]{LOT2}. Note that $\mathcal{Z}$ contains three Reeb chords $\rho_1, \rho_2, \rho_3 $ which may be used to generate a path algebra as described in the previous section. 

     

\begin{figure}
\begin{tikzpicture}[scale=.5]

\draw[thick, ->] (0,3) arc (90:-270:3cm);

\coordinate (1) at (1,2.8); 
\coordinate (2) at (1,-2.8);
\coordinate (3) at (3,0);
\coordinate (4) at (-2,-2.2);

 \draw[red][-] (1) to[out=15,in=-15,looseness=3.5] (2) ;
\draw[red][-] (3) to[out=-45,in=-60,looseness=3.5] (4) ;

\node at (-3.5,0) {$z$};
\node at (6,0) {$\alpha_1^a$};
\node at (1,-4.7) {$\alpha_2^a$};
\node at (1.8,1.8) {$\rho_1$};
\node at (1.8,-1.8) {$\rho_2$};
\node at (0,-2.8) {$\rho_3$};
\node at (-3,0)[circle,fill,inner sep=1pt]{};
\node at (1,2.8)[circle,fill,inner sep=1pt]{};
\node at (1,-2.8)[circle,fill,inner sep=1pt]{};
\node at (3,0)[circle,fill,inner sep=1pt]{};
\node at (-2,-2.2)[circle,fill,inner sep=1pt]{};

\end{tikzpicture}
\caption{ Pointed matched circle $\mathcal{Z}$ }
\label{figure2}
\end{figure}
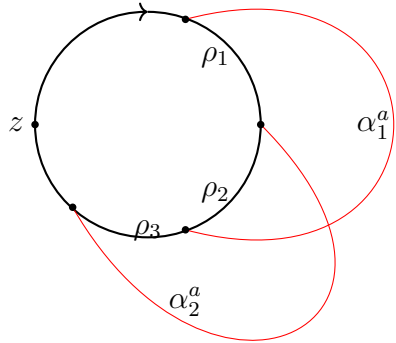

 The orientation of $F$ is given by $(\alpha^{a}_1,\alpha^{a}_2)$.
 If $\phi$ preserves orientation, then $(Y, \phi)$ is said to be of type $A$. Otherwise, it is said to be type $D$.
 
   A bordered 3-manifold $(Y, \phi)$  can be represented by a bordered Heegaard diagram.
   A bordered Heegaard diagram $\mathcal{H}$, in the torus boundary case, is a tuple
$$(\overline{\Sigma}; \{{\alpha}^{c}_1 \ldots {\alpha}^{c}_{g-1}\};\{ {\alpha}^{a}_1,{\alpha}^{a}_2 \};\{\beta_1 \ldots \beta_g \},z)$$
whose constituents are described as follows:
 \begin{enumerate}
 \item $\overline{\Sigma}$ is a compact, connected, oriented surface  of genus $g$ whose boundary is also connected.
\item  Two sets $\boldsymbol{\alpha}^{c} := \{{\alpha}^{c}_1 \ldots {\alpha}^{c}_{g-1}\}$ and $\boldsymbol{\beta} := \{\beta_1 \ldots \beta_g \}$  of pairwise disjoint circles that lie in the interior of $\overline{\Sigma}$.
\item   A set $\bs{\alpha}^a := \{ {\alpha}^{a}_1,  {\alpha}^{a}_2 \} $ that consists of two disjoint properly embedded arcs in $\overline{\Sigma}$  with 
boundary on $\partial\overline{\Sigma}$.
\item  A point $z$ on $\partial\overline{\Sigma}$ which is distinct from the endpoints of the two arcs ${\alpha}^{a}_1$ and ${\alpha}^{a}_2$. 
\end{enumerate}
 Furthermore, $\boldsymbol{\alpha}^{c}$ and $\boldsymbol{\alpha}^{a}$ are required to be disjoint, and $\overline{\Sigma } - \boldsymbol{\alpha}^a \cup \boldsymbol{\alpha}^c $ and $\overline{\Sigma} - \boldsymbol{ \beta} $ are connected.
 We denote $\boldsymbol{\alpha}^a \cup \boldsymbol{\alpha}^c $ by $\bs{\alpha}$.
 
  The manifold $Y$ can be recovered from  $\mathcal{H}$ by
attaching three dimensional $2$-handles to $\overline{\Sigma} \times I$ along the $\boldsymbol{{\alpha}}^{c}$
circles in $\overline{\Sigma} \times \{0\}$ and the $ \boldsymbol{\beta}$ circles in $\overline{\Sigma} \times \{1\}$, see \cite[Construction 4.6]{LOT2} for details. The
parametrization $\phi$ of $\partial Y$ is determined by the pointed matched circle $(\partial \overline{\Sigma}, {\alpha}^{a}_1,{\alpha}^{a}_2,z)$,
and $\partial \overline{\Sigma}$ is endowed with the  boundary orientation. If $(\partial \overline{\Sigma}, {\alpha}^{a}_1,{\alpha}^{a}_2,z)$
 is identified with $\mathcal{Z}$, then $\phi$ is orientation-preserving, and $\mathcal{H}$  gives rise to a  bordered $3$-manifold
$(Y, \phi)$ of type $A$. Otherwise, we have an identification of  $(\partial \Sigma, {\alpha}^{a}_1, {\alpha}^{a}_2,z)$ 
 with $-\mathcal{Z}$, and $\mathcal{H}$ yields a bordered
$3$-manifold of type $D$.  An example of a type $D$ bordered $3$-manifold is given in Figure \ref{figure3} (cf. \cite[Figure 3]{BW}).

     %

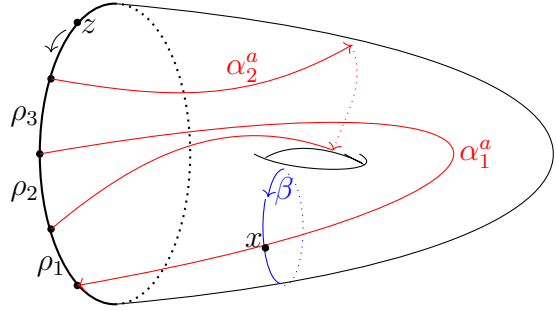
\begin{figure}
\begin{tikzpicture}[scale=.5]

\coordinate (1) at (0,4);
\coordinate (2) at (0,-4);

\coordinate (3) at (-1.3, 3.3);
\coordinate (4) at (-1.7, 2.7);

\draw[->] (3) to[bend right] (4);


\draw[thick, -] (0,4) arc (90:270:2cm and 4cm);
\draw[dotted, thick, -] (0,-4) arc (270:450:2cm and 4cm) ;


\draw (1) to[out=-5,in= 5,looseness= 5] (2) ;

\node at (-.7,3.4) {$z$};
\node at (-1,3.5)[circle,fill,inner sep=1pt]{};
 
\node at (-1.7,2)[circle,fill,inner sep=1pt]{};
\node at (-2.4,1) {$\rho_3$};
\node at (-2,0)[circle,fill,inner sep=1pt]{};
\node at (-2.4,-1) {$\rho_2$};
\node at (-1.7,-2)[circle,fill,inner sep=1pt]{};
\node at (-1.7,-3) {$\rho_1$};
\node at (-1,-3.5)[circle,fill,inner sep=1pt]{};

\draw[-] (3.7,0) to[out=-30,in=330,looseness=1] (6.5,0);
\draw[-] (4,-.1) to[out=30,in=-30,looseness=2] (6.1,0);

\draw[blue][->] (4.5,-.4) arc(90:150:.5cm and 1.5cm);
\draw[blue][-] (4,-1.2) arc(150:270:.5cm and 1.5cm);
\draw[blue][dotted,-] (4.5,-3.5) arc(270:450:.5cm and 1.5cm);

\draw[red][->] (-2,0) to[out=10,in=10,looseness=10] (-1,-3.5); 

\node at (4,-2.5)[circle,fill,inner sep=1pt]{};
\node at (3.7,-2.3) {$x$};
\node[blue] at (4.5,-1) {$\beta$};
\node[red] at (9.6,0) {$\alpha_1^a$};

\draw[red][->] (-1.7,2) to[out=-10,in=-150,looseness=1] (6.3,2.9); 

\draw[red][dotted,->] (6.3,2.9) to[out=-70,in=70,looseness=1] (5.8,0.1); 

\draw[red][-] (-1.7,-2) to[out=40,in=160,looseness=1] (5.8,0.1);

\node[red] at (3.4,2.3) {$\alpha_2^a$};

\end{tikzpicture}
\caption{ A genus $1$ bordered Heegaard diagram }
\label{figure3}
\end{figure}

   Consider a bordered Heegaard diagram $\mathcal{H}$ representing a bordered $3$-manifold $(Y, \phi)$. Then, bordered Floer theory \cite{LOT1, LOT2} associates a
type $A$ structure $(\widehat{CFA}(\{\mathcal{H}\})$, $\{{m}_k\}^{\infty}_{k=1})$
 to $\mathcal{H}$ if $(Y, \phi)$ is of type $A$. Similarly, it associates a type $D$ structure $(\widehat{CFD}(\{\mathcal{H}\}, \delta_1))$ to $\mathcal{H}$
if $(Y, \phi)$ is of type $D$. As  vector spaces over $\mathbb{Z}_2$ , $\widehat{CFA} (\{\mathcal{H}\})$ and $\widehat{CFD}(\{\mathcal{H}\})$ are generated by $g$-tuples 
of points  $\bs{x}$ in $\bs{\alpha} \cap \bs{\beta}$,  such that there is one point on each $\bs{\alpha}^{c}$ and
	 $\bs{\beta}$ circle, and there is one point 
on one of the $\bs{\alpha}^{a}$ arcs.

 There are natural right and left actions of the  subalgebra $\mathcal{I}$ of the path algebra corresponding to the pointed matched circle $\mathcal{Z}$ on  $\widehat{CFA}(\mathcal{H})$. The right action of $i_1$  preserves a generator $\bs{x}$ or annihilates  it depending on whether it lies on the  ${\alpha}^a_1$ or the  ${\alpha}^a_2$  arc, respectively. The left action of $i_1$ is just the opposite. See \cite[p.8]{BW} for precise formulae.

The type $A$ and type $D$ structure maps
$$ m_k \colon  \widehat{CFA} (\mathcal{H})\otimes_{\mathcal{I}}\mathcal{A} \ldots \otimes_{\mathcal{I}} \mathcal{A}
\rightarrow \widehat{CFA} (\mathcal{H}) $$ 
and
$$\delta_1 \colon \widehat{CFD}(\mathcal{H}) \rightarrow \mathcal{A} \otimes  \widehat{CFD}(\mathcal{H})$$
are defined by counting certain $J$-holomorphic curves in $\Sigma \times [0, 1] \times \mathbb{R}$, for a sufficiently
nice almost complex structure $J$ on $\Sigma \times [0, 1] \times \mathbb{R}$, where $\Sigma$ is the interior of $\overline{\Sigma}$. Details can
be found in \cite[Chapters 6-7]{LOT2}.  The type $A$ and type
$D$ structures $\widehat{CFA}(\mathcal{H}, \{m_k\}^{\infty}_{k=1})$
 and $\widehat{CFD}(\mathcal{H},\delta_1)$ are  independent of the choice of $J$ up to homotopy equivalence. Thus, they produce invariants of $\mathcal{H}$. Moreover, different bordered Heegaard diagrams for equivalent
bordered $3$-manifolds produce homotopy equivalent bordered invariants. Thus, one  obtains an invariant of a bordered $3$-manifold $(Y, \phi)$ that depends only on its equivalence class. If $(Y, \phi)$ is
of type $A$,  the invariant is denoted by $\widehat{CFA} [(Y, \phi)]$, and if $(Y, \phi)$ is of type $D$, it is denoted by $\widehat{CFD}[(Y, \phi)]$.

The following basic example is taken from \cite{BW}. Consider $D^{2} \times  S^{1}$
 having boundary parametrization $\psi \colon F \rightarrow \partial(D^{2} \times S^{1})$ given by ${\alpha}^{a}_1 \rightarrow {1} \times S^{1}$
 and ${\alpha}^{a}_2 \rightarrow \partial D^{2} \times {1}$.
 Then from the bordered Heegaard diagram in \ref{figure3} (cf.
 \cite[Figure 3]{BW}) for $(D^{2} \times S^{1},\psi)$
 it follows that $\widehat{CFD} [(D^{2} \times S_1
, \psi)]$ is given by the graph in Figure \ref{figure4} (cf. Figure 4 in \cite{BW}) below. This 
will serve as a stepping stone for the corresponding graph (Figure \ref{figure5}) when there is 
orbifold singularity along the core $S^1$.

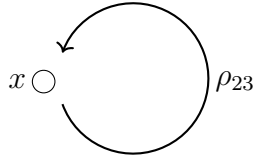
\begin{figure}
\begin{tikzpicture}[scale=.5]

\draw[-]  (0,0) arc (0:360:.3cm);
\node at ((-1,0) {$x$};

\draw[thick, ->] (0.2,-.6) arc(-160:160:2cm);

\node at (4.8,0) {$\rho_{23}$};

\end{tikzpicture}
\caption{ Type $D$ structure for $ (D^2 \times S^1, \psi)$ }
\label{figure4}
\end{figure}

 A type $DA$ structure is a $\mathbb{Z}_2$-vector space $N$ with the structure of an
$(\mathcal{I, I})$--bimodule, together with maps 
 $${\delta}^{k}_1:  N\otimes \mathcal{A}^{\otimes k-1} \rightarrow \mathcal{A}\otimes N \,.$$
See \cite{LOT1}, Definition 2.2.43, for details.
As in the case of type $D$ and $A$ structures, it is possible to define the box tensor product of a type $DA$ structure with a type $D$ structure, or that
 of a type $A$ structure with a type $DA$ structure, when at least one of the
factors is bounded  (cf. \cite[Definition 2.3.9]{LOT1}).
 If the parametrization of the boundary of $(Y, \phi)$ is changed, the bordered invariants
$\widehat{CFA}[(Y, \phi)]$ and $\widehat{CFD} [(Y, \phi)]$ change by 
an orientation-preserving homeomorphism $\psi$ of the model torus $F$. Then, by \cite[Theorem 2]{LOT1}, there exists a type $DA$
structure $\widehat{CFDA} (\psi)$ so that
$$\widehat{CFDA} (\psi) \boxtimes \widehat{CFD} [(Y, \phi)] \cong \widehat{CFD} [(Y, \phi \circ \psi)]$$
 as type D structures over $\mathcal{A}$, and
 $$\widehat{CFA} [(Y, \phi)]\boxtimes  \widehat{CFDA}(\psi) \cong \widehat{CFA}[(Y, \phi \circ \psi)]$$
 as type $A$ structures over $\mathcal{A}$.
 
 Consider a type $A$ bordered $3$-manifold $ (Y_1, \phi_1)$ and a type $D$ bordered $3$-manifold $(Y_2, \phi_2)$. Construct a closed, oriented, smooth 3-manifold $Y$ by gluing $Y_1$ and $Y_2$ along
their boundaries via the homeomorphism $\phi_2 \circ {\phi_1}^{-1} \colon \partial Y_1 \rightarrow \partial Y_2$. Associate the bordered invariants $\widehat{CFA} [(Y_1, \phi_1)]$ and $\widehat{CFD} [(Y_2, \phi_2)]$ to $(Y_1, \phi_1)$ and $(Y_2, \phi_2)$, respectively. 
 To $Y$,  associate the hat flavor of Heegaard Floer homology $\widehat{HF}(Y )$. The pairing theorem of bordered Heegaard Floer theory 
 implies that if at least one of the bordered invariants is bounded, then they determine $\widehat{HF}(Y)$ as follows:
\begin{equation}
\widehat{HF}(Y) =  {H}_{*}(\widehat{CFA}[Y_1] \boxtimes \widehat{CFD}[Y_2]) \,.
\end{equation}
We follow a similar approach to define Heegaard Floer invariants for closed $3$-orbifolds in the sequel. 

\section{Orbifold Heegard Floer for one component knot singularity}

In this section we review the construction of orbifold Heegard Floer homology by Wong \cite{BW} in the case where the singular locus has exactly one component. 
 
Let $Y^{orb}$
 be a compact, connected, oriented $3$-dimensional orbifold whose
singular locus ia a knot $K$  with cyclic singularity of order $n$. Fix a neighbourhood $N$ of $K$ modelled on
 $(D^{2} \times S^{1})/ \mathbb{Z}_n$
 and
an orientation-preserving homeomorphism $$\phi_N \colon (D^{2} \times S^{1})/\mathbb{Z}_n\rightarrow N \,.$$
Here the generator of $\mathbb{Z}_n$ acts on $D^2$ through a rotation by $2\pi/n$, and trivially on $S^1$. 
 This induces an orientation-preserving parametrization of the boundary
$$\phi_{\partial {N}}\colon  \partial(D^{2} \times S^{1})/\mathbb{Z}_n \rightarrow \partial N \,. $$

There is a natural orientation-reversing identification of the oriented torus $F$ associated to
the pointed matched circle $\mathcal{Z}$ from Figure \ref{figure2} 
with $\partial((D^{2} \times S^{1})/\mathbb{Z}_n)$,
 taking ${\alpha}^{a}_1$
to the longitude
$\{1\}\times S^{1}$
 and ${\alpha}^{a}_2$
to the meridian $\partial D^{2}/\mathbb{Z}_n \times \{1\}$. Using this, we regard $\phi_{\partial N}$ as an 
orientation-reversing parametrization of $\partial N$ by $F$.
The complement of the singular neighbourhood $N$ in $Y^{orb}$ is a
$3$-manifold E with torus boundary. Using the orientation-reversing parametrization $\phi_{\partial N}$ of
$\partial N$, define the following orientation preserving parametrization $\phi_{\partial E}$ of $\partial E$ 
\begin{equation}
\phi_{\partial E} := id \circ \phi_{\partial N} \colon F \rightarrow \partial E.
\end{equation}
  Then $E$, together with $\phi_{\partial E}$, forms a type $A$ bordered $3$-manifold. To $(E, \phi_{\partial E})$ we associate
the type $A$ structure $\widehat{CFA}[(E, \phi_{E})]$ coming from bordered Floer theory.
Generalizing the type $D$ structure $\widehat{CFD}[(D^{2} \times S^1,\psi )$ in Figure \ref{figure4}, we associate to the
singular piece $N$ the type $D$ structure $D_N$ in Figure \ref{figure5} (cf. Figure 5 \cite{BW}).

\begin{figure}
\begin{tikzpicture}[scale=.5]

\draw (2,2) circle(.3);
\node at (2.7,2.5) {$x_n$};

\draw[->] (1.5,2.5) to[out=160,in=20,looseness=1.5] (- 1.5,2.5);
\node at (0,3.5) {$\rho_{23}$};

\draw (-2,2) circle(.3);
\node at (-2.7,2.5) {$x_1$};

\draw[->] (-2.5,1.5) to[out= 230,in= 75,looseness=1.5] (- 3.8,-.5);
\node at (-4.2,.6) {$\rho_{23}$};

\draw (-4,-1) circle(.3);
\node at (-4.7,-1.5) {$x_2$};

\draw[->] (3.8, -.5) to[out=100,in=-20,looseness=1.5] ( 2.5,1.5);
\node at (4.3,.8) {$\rho_{23}$};

\draw (4,-1) circle(.3);
\node at (4.9,-1.5) {$x_{n-1}$};

\draw[->] (-3.8, -1.5) to[out=-60,in=150 ,looseness=1] ( -2.2,-2.8);
\node at (-3.1,-3) {$\rho_{23}$};

\node at (-1.7,-3)[circle,fill,inner sep=1pt]{};

\node at (-1,-3.2)[circle,fill,inner sep=1pt]{};

\node at (-0.3,-3.3)[circle,fill,inner sep=1pt]{};

\node at (.5,-3.3)[circle,fill,inner sep=1pt]{};

\node at (1.2,-3.2)[circle,fill,inner sep=1pt]{};

\draw[->] (1.6, -3.1) to[out=15,in=270,looseness=1] ( 3.7,-1.8);
\node at (2.9,-3.3) {$\rho_{23}$};




\end{tikzpicture}
\caption{ The structure $D_N$  }
\label{figure5}
\end{figure}
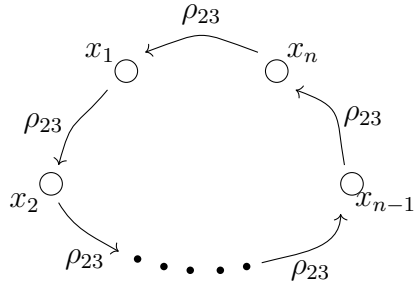
 
Similar to $\widehat{CFD} [{D}^{2} \times S^1,\psi]$, $D_N$ arises naturally from an orbifold bordered Heegaard
diagram  for $(N, \phi_{\partial_N} )$; see Figure \ref{figure6} (cf. Figure 6 \cite{BW}). 

   %

\begin{figure}
\begin{tikzpicture}[scale=.5]

\coordinate (1) at (0,4);
\coordinate (2) at (0,-4);

\coordinate (3) at (-1.3, 3.3);
\coordinate (4) at (-1.7, 2.7);

\draw[->] (3) to[bend right] (4);


\draw[thick, -] (0,4) arc (90:270:2cm and 4cm);
\draw[dotted, thick, -] (0,-4) arc (270:450:2cm and 4cm) ;


\draw (1) to[out=-5,in= 5,looseness= 5] (2) ;

\node at (-.7,3.4) {$z$};
\node at (-1,3.5)[circle,fill,inner sep=1pt]{};
 
\node at (-1.7,2)[circle,fill,inner sep=1pt]{};
\node at (-2.4,1) {$\rho_3$};
\node at (-2,0)[circle,fill,inner sep=1pt]{};
\node at (-2.4,-1) {$\rho_2$};
\node at (-1.7,-2)[circle,fill,inner sep=1pt]{};
\node at (-1.7,-3) {$\rho_1$};
\node at (-1,-3.5)[circle,fill,inner sep=1pt]{};

\draw[-] (3.7,0) to[out=-30,in=330,looseness=1] (6.5,0);
\draw[-] (4,-.1) to[out=30,in=-30,looseness=2] (6.1,0);

\draw[blue][->] (4.5,-.4) arc(90:150:.5cm and 1.5cm);
\draw[blue][->] (4,-1.2) arc(150:270:.5cm and 1.5cm);

\draw[blue][dotted,->] (4.5,-3.5) to[out=65,in=-75,looseness=1] (5,-.4);

\draw[blue][->] (5,-.4) to[out=-140,in=110,looseness=.8] (5,-3.3);

\draw[blue][dotted,->] (5,-3.3) to[out=65,in=-50,looseness=1.5] (4.5,-.4);

\node at (4.7,-2.3)[circle,fill,inner sep=1pt]{};

\node at (5.3,-2.6) {$x_2$};


\draw[red][->] (-2,0) to[out=10,in=10,looseness=10] (-1,-3.5); 

\node at (4,-2.5)[circle,fill,inner sep=1pt]{};
\node at (3.5,-2.2) {$x_1$};
\node[blue] at (3.5,-1) {$\beta$};
\node[red] at (9.6,0) {$\alpha_1^a$};

\draw[red][->] (-1.7,2) to[out=-10,in=-150,looseness=1] (6.3,2.9); 

\draw[red][dotted,->] (6.3,2.9) to[out=-70,in=70,looseness=1] (5.8,0.1); 

\draw[red][-] (-1.7,-2) to[out=40,in=160,looseness=1] (5.8,0.1);

\node[red] at (3.4,2.3) {$\alpha_2^a$};

\end{tikzpicture}
\caption{ A genus $1$ orbifold bordered Heegaard diagram for $n=2$ }
\label{figure6}
\end{figure}
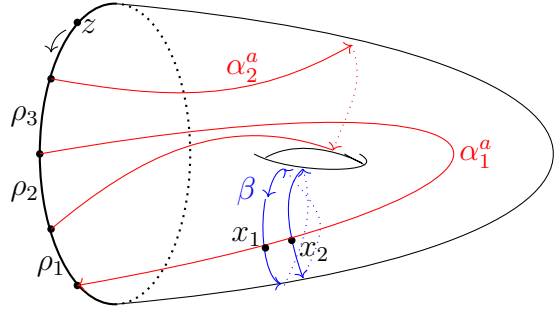
    
Here we begin with a $\mathbb{Z}_n$-equivariant (quotient) torus that has
been punctured once on the boundary $\partial N$, together with two properly embedded arcs ${\alpha}^{a}_1$
and ${\alpha}^{a}_2$.
 As before, $\beta$ represents a meridian of an honest handlebody.
However, $\beta$  is now  immersed in the punctured $\mathbb{Z}_n$-equivariant torus, wrapping $n$ times
around ${\alpha}^{a}_2$.
 In other words, $\beta$ represents one full meridian, while ${\alpha}^{a}_2$
represents a meridian of the quotient
 solid torus $N$ under the $\mathbb{Z}_n$-action, that is, an $n^{th}$ of a full meridian. The generators $x_i$ of the type
$D$ structure $D_N$ correspond to the intersection points of $\beta $  and  ${\alpha}^{a}_1$.

\subsection{Definition}
 Let $\widehat{CFO}[(Y^{orb}]$
 be the box tensor product of $\widehat{CFA}[(E,\phi_E)]$ and  $D_N$. We define
$\widehat{HFO}[(Y^{orb}]$
 to be the homology of $\widehat{CFO}[(Y^{orb}]$.
\begin{remark}\label{nice}
 This makes sense only if the type $A$ structure is bounded. We will restrict to structures coming from nice diagrams, cf. \cite[Corollary 8.6]{LOT2}, which have $m_k=0$
 for $k \geq 3$.
\end{remark}

Wong  has shown in \cite{BW} the above homology  is well defined and is a diffeomorphism
invariant.

\section{Orbifold Heegaard Floer for two component singular locus}

We now address the case of two component singular locus.
Consider the orbifold $Y$ with cyclic singularities along the knots $K_1$ and $K_2$  of order $n_1$ and $n_2 $, respectively. We denote this data by $(Y,K_1,n_1,K_2,n_2)$. 
There exist neighbourhoods $N_i$ of $K_i$ such that there is a parametrization $\phi_i$ of $N_i$ by $({D}^{2} \times S^{1})/ \mathbb{Z}_{n_i}$ where the action on $ D^{2} \times S^{1}$ is the standard rotation action on $D^{2}$ with axis $S^{1}$.  Let $Y_i$ denote the orbifold with boundary $Y-N_i$, for $i=1,2$.

 Let  $Y^{M}_1$ be the manifold with boundary obtained from $Y_1$ by changing the parametrization of $N_2$ to $D^{2} \times  S^{1}$, that is, by replacing $N_2$ by its branched cover in the definition of $\phi_2$.  Similarly, $Y^M_2$ is defined to be the manifold with boundary obtained by replacing $N_1$ in $Y_2$ by its branched cover appearing in the definition of $\phi_1$.
 A parametrization for the boundary of $Y^{M}_i$ is obtained by the process outlined in Section 3, making $Y^M_i$ into a bordered 3-manifold. 
 Let $Y^{M}$ be the manifold obtained by replacing both $N_1$ and $N_2$ with their corresponding smooth branched covers. 

 \subsection{More $A$ structures}
 Here we define some more $A$ structures in order to define the orbifold Floer homology. 
 But, first, we mention the path algebra  required in torus boundary case, see \cite{LOT2}, Chapter 11, for details.
 The differential $\mu_1$ of the path algebra is zero and also $\mu_i = 0$ for $i > 2$. So, only $\mu_2$ is nonzero and considering it as an ordinary multiplication, the nonzero products are as follows.
 \begin{equation}
  \rho_1 \rho_2 = \rho_{12}
 \end{equation}
 \begin{equation}
 \rho_2 \rho_3= \rho_{23}
  \end{equation}
 \begin{equation}
 \rho_1 \rho_{23}=\rho_{123}
  \end{equation}
  \begin{equation}
  \rho_{12} \rho_{3}= \rho_{123}
 \end{equation}

 Our goal is to construct an $A$-structure  $\widehat{CFA}^{orb} [Y_2,\overline{m}_k]$ such that we have the following isomorphism of box tensor products, 
 $$\widehat{CFA}^{orb}[Y_2] \boxtimes D_{1} \cong \widehat{CFA}[Y^{M}_1]\boxtimes D_{n_1} \,.$$  
 
 As vector spaces,
 $$\widehat{CFA}^{orb}[Y_2,\overline{m_k}] \cong {\bigoplus}^{n_1}_{j=1} \widehat{CFA}[Y^{M}_1,{m}_k]\,.$$  
 
 Denote the vector $(0, \ldots, 0, y, 0, \ldots,0)$, where $y$ occupies the $j$-th position, by $(y : j)$. Also, for an integer $j$,  define 
 \begin{equation}\label{boxbra}
  [j] := \left\{ \begin{array}{ll} n_1 & {\rm if }\, n_1 \, {\rm divides} \, j \\
  j \,{\rm mod}\, n_1 \, & {\rm otherwise} \end{array} \right. \end{equation}

 The operator 
 $\overline{m}_k$ is defined by 
 \begin{equation}
  \overline{m}_k( (y:j),a_1, \ldots
a_{k-1})= ( m_k(y,a_1,\ldots a_{k-1}) : [j+l]) \,,
  \end{equation}
 where $l$ is the number of times the elements of the set $\{\rho_{23},\rho_{3},\rho_{123}\}$ occur in $(a_1, a_2, \ldots, a_{k-1})$. 
 In other words, we shift by one position for each $\rho_{23},\rho_{3},\rho_{123}$ entry and we do not offset any entry for other elements of the algebra. 
 
 \begin{lemma}\label{composition laws} 
 The operators $\overline{m}_k$ statisfy the the structural equations \eqref{compat} of an $A$ structure.  
 
 \end{lemma}
 
 \begin{proof} It suffices to check that the shift of
 	 $$ {m}_{k-j+1}({m}_{j}((y:i),a_1\otimes a_2 \ldots \otimes a_{j-1}),a_j \otimes \ldots \otimes a_{k-1}) $$ in 
 $$ \overline{m}_{k-j+1}(\overline{m}_{j}((y:i),a_1\otimes a_2 \ldots \otimes a_{j-1}),a_j \otimes \ldots \otimes a_{k-1}) $$
 matches the shift of  $${m}_{k-1}((y:i),a_1 \otimes\ldots a_j\otimes a_j a_{j+1}\otimes a_{j+2}\ldots  \otimes a_{k-1})  \, $$ in
  $$\overline{m}_{k-1}((y:i),a_1 \otimes\ldots a_j\otimes a_j a_{j+1}\otimes a_{j+2}\ldots  \otimes a_{k-1})  \, .$$

\noindent This is straightforward except when $a_j a_{j+1}$  is nontrivial and a $\rho_{23}$ is created or destroyed through this product. These can be verified on a case by case basis. For instance, when $a_j=\rho_2$ and $a_{j+1} = \rho_3$, their contribution in the shift of the first term is one as $\rho_3$ contributes but $\rho_2$ does not.   
In the second term also the product $\rho_2 \rho_3 = \rho_{23}$ contributes one to the shift. Similarly, consider the
case $a_j = \rho_1$ and $a_{j+1} = \rho_{23}$. Here, only $\rho_{23}$ contributes a shift to the first term and $\rho_{123}$ has a matching contribution to the second term.
 \end{proof} 

 
  

 \begin{lemma}\label{equiv2} There is a natural isomorphism of box tensor products
  $$T:   \widehat{CFA}^{orb}[Y_2] \boxtimes D_{1}  \longrightarrow \widehat{CFA}[Y^{M}_1] \boxtimes D_{n_1} \, .$$  
 \end{lemma}

\begin{proof} Define 
$$T ((y:j) \otimes x) = y \otimes x_j \,.$$ 
 Then $T$ naturally extends to a vector space isomorphism by linearity. Observe that
 
 $$ 
 \begin{array}{ll}  & T \left( \delta^{\boxtimes} ((y:j) \otimes x) \right)   \\
  & \\
 =  & T \left( \sum_{k=0}^{\infty} (\overline{m}_{k+1} \otimes {\rm id} ) ( (y:j) \otimes \delta_k(x))  \right) \\
 &\\
 
   =  & T \left( \sum_{k=0}^{\infty} (\overline{m}_{k+1} \otimes {\rm id} ) ( (y:j) \otimes \rho_{23}^{\otimes k} \otimes x )  \right) \\
 
 &\\
 
  =  & T \left( \sum_{k=0}^{\infty} (\overline{m}_{k+1}  ( (y:j) \otimes \rho_{23}^{\otimes k})   \otimes x)  \right) \\
  
  &\\
  
  = & T \left( \sum_{k=0}^{\infty} ({m}_{k+1}  ( y \otimes \rho_{23}^{\otimes k}): [j+k])   \otimes x \right) \\
 
 &\\
 
 = &  \sum_{k=0}^{\infty}  {m}_{k+1}  ( y \otimes \rho_{23}^{\otimes k})  \otimes x_{[j+k]}  \,.
 
 \end{array}
 $$
 
\noindent  On the other hand, 

$$ \begin{array}{ll}  &   \delta^{\boxtimes} (T ((y:j) \otimes x) )   \\

&\\

= & \delta^{\boxtimes} (y \otimes x_j)    \\

&\\ 

= & \sum_{k=0}^{\infty} ({m}_{k+1} \otimes {\rm id}) (y \otimes \delta_k(x_j) ) \\

&\\

= & \sum_{k=0}^{\infty} ({m}_{k+1} \otimes {\rm id}) (y \otimes \rho_{23}^{\otimes k} \otimes x_{[j +k ]} ) \\

&\\

= &  \sum_{k=0}^{\infty}  {m}_{k+1}  ( y \otimes \rho_{23}^{\otimes k})  \otimes x_{[j+k]}  \,. \\

\end{array}
  $$
 \medskip
 
 \noindent Thus $T $ commutes with $\delta^{\boxtimes}$. This completes the proof.


\end{proof}


 

\begin{defn} We define an orbifold Heegaard Floer homology of $Y^{orb}$ by
	$$\widehat{HFO}^{12}_{*}(Y^{orb}) =  H_{*}(\widehat{{CFA}}^{orb}[Y_2] \boxtimes D_{n_2})\,. $$
\end{defn}

Switching the roles of $N_1$ and $N_2$ in the above constructions, we may similarly define $\widehat{CFA}^{orb}[Y_1]$ and a corresponding orbifold Heegaard Floer homology $\widehat{HFO}^{21}_{*}(Y^{orb})$ of $Y^{orb}$ by
\begin{equation}
\widehat{HFO}^{21}_{*}(Y^{orb}) =  H_{*}(\widehat{CFA}^{orb}[Y_1] \boxtimes D_{n_1}) 
\end{equation}

\subsection{$DA$ structures and invariance}
 A $DA$ structure over an algebra with a generating set $N$ is defined through a series of maps
 $$ {\delta}^{j}_1: N \otimes {A}^{j-1}   \rightarrow A \otimes N \,.$$

These maps induce a series of maps ${\delta}^{j}_k$, $k \geq 1$, defined on page 555 of \cite{LOT1} (in their notation  $k$ is a prefix, we follow the notation of \cite{BW}).  Since we restrict to nice diagrams, we only need the $k=1$ case.

A box product of an $A$ structure with a $DA$ structure
yields an $A$ structure with differentials $m^{box}_i$ given by
\begin{equation}
m^{box}_i ={ \sum_{j=0}}^{\infty} (m_{j+1}\otimes \mathbb{I}_{DA}) \circ (\mathbb{I}_A \otimes {\delta}^{i}_j)
\end{equation}
where ${{\delta}^{i}}_0 = \mathcal{I}$, $\mathbb{I}_A$ and $\mathbb{I}_{DA}$ are the identity maps on the $A$ structure and $DA$ structure respectively, and the $m_j$ are the differentials of the initial $A$ structure.

\begin{theorem}
 The homology $\widehat{HFO}^{12}_{*}(Y^{orb})$ is well defined, and is a diffeomorphism invariant.
\end{theorem}

\begin{proof}  Suppose the parametrizations are changed by a factor $\psi$
 where $\psi$ is a  Dehn twist around $\alpha_2^{a}$.
  The $A$ structure  $\widehat{CFA}^{orb}[Y_2]$
  depends on the $A $ structure of
 $\widehat{CFA}[Y^{M}_1]$. Denoting the $A$ structures corresponding to the two parametrizations
 by $\widehat{CFA}^{orb}[Y_2,\phi]$ and $\widehat{CFA}^{orb}[Y_2,\phi \circ \psi^{-1}]$,  
 we show that the homology does not depend on the $\psi$ factor.

Let $m^{box}$ be the differential of $\widehat{CFA}[Y^{M}_1](\phi) \boxtimes  CFDA (\psi)$.  
For a  Dehn twist $\psi$, we define an $A$ structure
$\widehat{CFA}^{orb}[Y_2,\phi ] \widehat{\otimes} CFDA[\psi]$
 with differential
 \begin{equation}\label{mbarbar}
\overline{\overline{m_i}}((x\otimes y: j), a_1, \ldots ,a_{i-1})=( m^{box}_{i}(x \otimes y,a_1 \ldots a_{i-1}): [j + l] )
 \end{equation}
 where  $l$ is number of terms in $(a_1, \ldots ,a_{i-1} )$ which belong to the set $\{\rho_3,\rho_{23},\rho_{123}\}$, and $[j+l]$ is defined according to \eqref{boxbra}. 
 
 We may choose  nice diagrams (see Remark \ref{nice}) so that from $m_3$ onwards the differentials of $\widehat{CFA}[Y^{M}_1](\phi)$  and $\widehat{CFA}[Y^{M}_1](\phi \circ \psi^{-1})$ are zero.
 Let $t_1$ be the equivalence between the $A$ structures  $\widehat{CFA}[Y_1^{M}](\phi \circ \psi^{-1})$ and  
 $ \widehat{CFA}[{Y_1}^{M}](\phi) \boxtimes CFDA(\psi)$.
 Now, define a morphism $T_1$ (cf. \cite[Definition 2.9]{LOT2})  between the two $A$ structures
$\widehat{CFA}^{orb}[Y_2,\phi \circ \psi^{-1}] $ and  $\widehat{CFA}^{orb}[Y_2,\phi ]\widehat{\otimes} CFDA(\psi)$
 such that
 \begin{equation}
 {T_1}_{i}( ( x :j), a_1,\ldots ,a_{i-1})=( {t_1}_i(x,a_1 \ldots a_{i-1}) : [j+l])
 \end{equation}
 where $l$ is defined as above. 
 
 Since $t_1$ is an isomorphism, it follows that $T_1$ is also an isomorphism.
However, some care is needed: Note that the shift in the differential $\overline{\overline{m_i}} $ of $\widehat{\otimes}$, and the shift calculated naturally after applying the differential
$m^{box}_i $, have some differences. 
When taking the differential $\overline{\overline{m_i}}$, the shift is determined by the $a_i'$s, see
\eqref{mbarbar}.
But in $m^{box}$,
 the $m_i'$s get factored by 
 ${\delta}^{i}_{j}$ which changes the arguments $a_i'${s}. This can affect the shifts. 
 We need to consider $m_1$ and $m_2$ only as noted earlier.
 First, consider  $m_2((x\otimes y :j), a_1)$.  In the case of $\widehat{\otimes}$, the shift is determined completely by $a_1$, but in the latter case the expression has terms like  $x\otimes{\delta}^{2}_1(y,a_1)$. 
 The ${\delta}^{2}_1(y,a_1)$ is equal to a sum of terms of the form $b_i\otimes z_i$, where $b_i \in A$. The shifts are now determined by the $b_i'$s. 
 Lemma 3.4 of \cite{BW} contains a list of the terms with nontrivial differentials.
 However, for a Dehn twist around $ \alpha_2^{a}$, we may restrict to the subalgebra of the path algebra generated by $\rho_{23}$.
  Then, there
 is only one term with one non-zero ${\delta}^{2}_1$ to consider, 
 $${\delta}^{2}_1( {\bf q} \otimes \rho_{23}) = \rho_{23} \otimes {\bf q} \, ,$$
 where ${\bf q}$ is one of the generators of $CFDA(\psi)$. 
 So, the shift is preserved. Finally, consider $m_1$.
 Then, there is a nontrivial  ${\delta}^{1}_1$ term, ${\delta}^{1}_1 ({\bf r}) = \rho_2 \otimes {\bf q} $.  As $\rho_2$ does not contribute to shifts under our assumptions, no problems arise.  
 
 
 Then, by the argument in \cite[Lemma 3.4]{BW}  using the edge reduction algorithm, 
 we have 
 \begin {equation}
 H_{*}( \widehat{CFA}^{orb}[Y_2,\phi \circ \psi^{-1}]\boxtimes  D_{n_2})=H_{*}{(({\widehat{CFA}^{orb}[Y_2,\phi ]\boxtimes CFDA[\psi ]}})\boxtimes D_{n_2}) \,.
 \end {equation}
 Since the left hand side $A$ structure  comes from a nice diagram and the right hand side $A$ structure is also bounded due to the restriction, the homology is well defined.
 

 It has been
shown in \cite{BW} that a general $\psi$ is isotopic to a product of Dehn twists. If it is exactly a product of Dehn twists, the above argument suffices. If it is an isotopy, then the $A$ structures on  $Y^{M}_1$ are bordered equivalent since the isotopy results in an ambient isotopy $F_t$, and $F_1$ gives the required diffeomorphism. The equivalence of the two $A$ structures on $Y^{M}_1$  implies the equivalence of two orbifold  $A$ structures  on $Y_2$ using relevant shifts.  This completes the proof of parametrization invariance.

To prove neighbourhood invariance, we again follow \cite{BW}. Neighbourhoods of the same knot  can be deformed to one another by an ambient isotopy $F_t$ and the diffeomorphism $F_1$ makes the bordered manifolds bordered equivalent. The equivalence in the orbifold $A$ structure follows by defining proper shifts.
%
\end{proof}

\begin{remark}
 Taking either $n_i=1$ by Lemma \ref{equiv2} our homology invariants coincide and equal the invariant of one singular component defined in \cite{BW} which generalizes the hat version of Heegaard Floer homology of manifolds. Also, since every bordered manifold
 has a nice diagram, see \cite[Proposition 8.2]{LOT2}, we can choose a bounded  $A$ structure.
 \end{remark}
 
 There are a limited number of situations when we are certain that the two invariants $\widehat{HFO}^{12}(Y^{orb})$ and $\widehat{HFO}^{21}(Y^{orb})$ actually are the same. 
 
 \begin{lemma}\label{equiv} 
  The bordered manifolds $Y^{M}_i$, $i = 1, 2$, are equivalent when they are both handlebodies. In this case, the rank of the two homology invariants,  $\widehat{HFO}^{12}(Y^{orb})$ and   $\widehat{HFO}^{21}(Y^{orb})$, coincide.
 \end{lemma}
 
\begin{proof}
 Under the hypothesis, ${Y^{M}_i}$ and $N_i$, $i=1, 2$, give two Heeegaard splittings of the same  manifold $Y^{M}$. The matrix of the Heegaard gluing map is entirely determined by the boundary  parametrization of $Y^{M}_i$ as $F$ comes from a common $D_1$ structure as described in the previous section. The gluing maps can be that of the sphere $S^3$,  $S^{2} \times S^{1}$, or a lens space $L(p,q)$, as only these manifolds allow genus one splittings.
 
 The gluing matrix with respect to the standard bases in the sphere case  has first column $(0,1)$ and second column
 $(-1,m)$, $m \in \mathbb{Z}$, since it is obtained by attaching a two cell to the longitudinal circle. Now,  since  the diffeomorphism
  $$ f_m: S^1 \times D^2 \longrightarrow S^1 \times D^2, \quad (e^{\theta_1},t e^{\theta_2}) \longmapsto (e^{\theta_1}, t e^{-m \theta _1+\theta_2}), \, 0 \leq t \leq 1, $$  maps
   the boundary longitude to $(1,-m)$, the bordered manifold $Y^{M}_1$ is equivalent to a bordered manifold with a parametrization which induces the same map at the boundary up to isotopy with $Y^{M}_2$. This means that ${\psi_1}^{-1} \psi_2$, where $\psi_i$ is the parametrization of the boundary of $Y^{M}_i$, is isotopic to the identity. Any map isotopic to identity in the torus can be extended to the whole solid torus. So, the parametrization $\psi_1$ of $Y^{M}_1$ is bordered equivalent to $\psi_2$ of $Y^{M}_2$.
 
 For $S^{2} \times S^{1}$ the first column is $(1,m)$ and the second
 column is $(0,1)$ since the two cell is attached along  the meridian. Again this is similar to the the sphere case as  $(1,m)$ are
 isotopic to $(1,0)$ in solid tori. So $Y^{M}_i$ are bordered equivalent in this case also.
 
 For the lens space $L(p,q)$, $p>1$ and $(q,p)=1$, the two cell is attached along
 $p\lambda + q\mu $ where $\lambda$ is the longitude and $\mu$
 the meridian. We will get exactly  one matrix for a given $(p,q)$  from the equivalent parametrizations if we restrict the  absolute values of the entries to be $\leq p $.
   
   But, for the cases when (i) $qq^{\prime}\cong 1 \,{\rm mod} (p)$, (ii) $q + q^{\prime} \cong 0 \, {\rm mod} (p)$, and (iii)
 $qq^{\prime} \cong -1 \, {\rm mod}\, (p)$, $L(p,q)$ and $L(p, q^{'})$
 are  diffeomorphic lens spaces. The lens space $L(p,q)$ is obtained by
 taking the quotient of the action of the cyclic group $\ZZ_p$ on $S^{3} =\{ (z_1, z_2) \in \CC^2:\, \mid z_1 \mid + \mid z_2 \mid =1  \}$ generated by the map
 $$ (z_1, z_2) \longmapsto ( e^{\frac{2\pi i}{p}}z_1, e^{\frac{2 \pi i \, q}{p}} z_2)\,.$$
 The splitting is given by two solid tori $ \mid z_1 \mid \leq \frac{1}{2}$ and 
 $\mid z_2 \mid \leq \frac{1}{2}$. The diffeomorphism from $L(p, q)$ to $L(p, q^{'})$
   in  case (i) is induced by the map $(z_1, z_2) \mapsto  (z_2, z_1)$. Since  this is an orientation reversing map, it 
   induces an orientation reversing map between $N_i$'s.   But this is not possible since they share the same $D_1$ parametrization.
   Similarly, in case (ii) we have the diffeomorphism $\arg(z_2) \mapsto -\arg(z_2)$. This is also orientation reversing and hence not possible. Similar reasoning holds in case (iii) as well.
   
    If we allow for larger values of $\mid q \mid >p$ , and let $q= q^{\prime} + mp$ where $0<q^{\prime} < p $
 then $p\lambda + q\mu = p(\lambda +m\mu) +q^{\prime}\mu$. This reduces
 to the previous case as $\lambda + m\mu$ is isotopic to $\lambda$
 in the solid torus. If $\mid q \mid < p$ and we do not restrict the entries,  the inverse transformation maps the meridian to $(p,q)$ and the longitude to $(m + kp, \, n+ kq)$ where $\mid m \mid < p$ and $ \mid n \mid < p $. So its inverse will
   map $(p,q)$ to the meridian $\mu$ and $(m,n)$ to $\lambda -k \mu$, which reduces to the previous cases as $\lambda -k\mu$
    is isotopic to the longitude $\lambda$  in a solid torus.
 
  Thus, we conclude that the two bordered manifolds ${Y}^{M}_{i}$,  $i=1, 2$, are equivalent.
  \end{proof}

  \begin{example}\label{example: lens space}
Think of $L(p, -q)$ as two copies of $D_2\times S^{1}$
 glued together. Singularize the two knots lying in each copy of $D^{2}\times S^{1}$ such that each knot is homotopic to $\{0\} \times S^{1}$. The copies of $D^{2}\times S^{1}$  will be $N_i$'s and $K_i$ is the core of $N_i$. Take $p \geq  2$, $1 \leq  q \leq p-1$, and
$gcd(p, q) = 1$. Then figure $14(A)$ of \cite{BW} gives a bordered Heegaard diagram for ${Y_i}^{M}$. The induced
type $A$ structure $\widehat{CFA} [{Y_i}^{M}, \phi_{\partial {Y_i}^{M}}]$  is shown in Figure $14(B)$ of \cite{BW}. It has $p$
generators $y_i$. Note that this $A$ structure is bounded. Thus, $\widehat{CFA}^{orb}[Y_2]$ is sum of $n_1$ copies of $\widehat{CFA}[{Y_1}^{M}]$  thus having $p n_1$ generators $y_i \otimes x_j$  and the differential
of its box product with $D_{n_2}$ is trivial. The homology 
$\widehat{HFO}^{12}_{*}(Y^{orb})$
is given by $$ \ZZ_2(  y_i \otimes x_j \otimes x_k^{\prime} : 1\leq i \leq p,\, 1\leq j \leq n_1,\, 1\leq k \leq n_2  ) \cong (\ZZ_2)^{p n_1 n_2} \,.$$
 We get the same result for the other invariant $\widehat{HFO}^{21}_{*}(Y^{orb})$  by Lemma \ref{equiv}.
  \end{example}
  
\section{Multiple component singular locus}

 
 First, consider the case when the orbifold locus consists of three knots. So, let $Y$ be a $3$-orbifold with cyclic singularities along knots $K_i$, $i=1,2,3$. Let $N_i$
 be a regular neighbourhood of $K_i$ parametrized by $(D^2 \times S^1)/\ZZ_{n_i}$.
 Let $Y_i = Y-N_i$.  Let $Y_{i}^{M}$ be the manifold with boundary obtained by replacing each $N_j$, $j \neq i$, with its smooth branched cover. 
 
 We can use $n_1$ copies of $\widehat{CFA}[Y_1^M]$, as in the previous section, to construct an $A$ structure 
 $\widehat{CFA}^{orb}[Y_{23}] $ such that 
 $$\widehat{CFA}^{orb}[Y_{23}] \boxtimes D_1 =  \widehat{CFA}[Y_1^M] \boxtimes D_{n_1}\,.  $$
 
 Then, we can use $n_2$ copies of $\widehat{CFA}^{orb}[Y_{23}] $ to construct an $A$ structure $\widehat{CFA}^{orb}[Y_{23,3}]$ that satisfies
 $$\widehat{CFA}^{orb}[Y_{23,3}] \boxtimes D_1 =  \widehat{CFA}^{orb}[Y_{23}] \boxtimes D_{n_2}\,.  $$
 
 Finally, we define the orbifold Heegaard Floer homology 
 $$ \widehat{HFO}^{123}_{*}(Y^{orb}) := H_{*}(\widehat{CFA}^{orb}[Y_{23,3}] \boxtimes D_{n_3}) \,.$$
 
  The invariant obtained above may depend on the choice of ordering on the three knots in the  construction. Note that the same procedure
 generalizes to the situation where the orbifold locus has $N$ components, whence we obtain $N!$ invariants corresponding to different orderings of the $N$ components. 
   However, we know that the ranks many of these homology groups coincide in some special cases.
 
 \begin{example}\label{sub:r-surg} Let $\epsilon$ be the knot invariant discussed in \cite[Subsection 2.4]{BW}. 
 Let $Y$ be a cyclic $3$-orbifold with singularities along $N$ knots of orders $n_1, \ldots, n_N$. Assume that the underlying manifold of $Y$ is obtained from $r$-surgery of a singular knot $K$  in $S^{3}$  with the knot invariant $\epsilon(K)=-1$  or $ 1$. Then for $(N-1)!$ out of the $N!$ orbifold Heegaard Floer homologies, the $A$ structures mainly depend on the $A$ structure of the complement of $K$. The rank of each of these is $\prod_{j=1}^N n_j$ times the rank of the Heegaard Floer homology of the underlying smooth manifold.



This may be argued, based upon \cite[Theorem 1.2]{BW}, as follows. Firstly, in \cite[Theorem 1.2]{BW}  the kernel and images of the boundary operator of the box products of the $D$ and $A$ structures in the smooth as well as the orbifold cases are given by nice generators, if the $A$ structures are obtained from the knot Floer chain complex of the knot $K$. In this case the structures are bounded. 
Consider first a single component singular locus. Then the result follows directly from \cite[Theorem 1.2]{BW}.
For the multiple component singular locus case, we start the iterative construction of $A$ structures by always taking $K_1$ to be the knot along which $r$-surgery has been performed. Then,  
the $A$ structure $\widehat{CFA}^{orb}[Y_{23 \ldots N}]$ will have $n_1$ times as many generators as $\widehat{CFA}[Y_{1}^M] $. The latter corresponds to the number of generators in the smooth case. After that,  we can choose any ordering of the remaining $(N-1)$ knot components. At each stage of the construction, the number of generators is multiplied by the order of singularity of the corresponding knot component. Thus, we have $(N-1)!$ invariants each of which have rank $\prod_{j=1}^N n_j$.
 \end{example}
 

\end{document}